\documentclass[runningheads]{llncs}
\usepackage{amsmath}
\usepackage{amssymb}
\usepackage{mathabx}
\usepackage[T1]{fontenc}
\usepackage{graphicx}
\usepackage{hyperref}
\usepackage{color}

\newtheorem{assumption}{Assumption}

\newcommand{\R}{\mathbb{R}}
\newcommand{\X}{\mathcal{X}}
\newcommand{\dd}{\mathrm{d}}

\newtheorem{algorithm}{Algorithm}
\begin{document}
\title{Maximum likelihood estimation for the $\lambda$-exponential family}
%
%
\author{Xiwei Tian\inst{1} \and
Ting-Kam Leonard Wong\inst{1}\orcidID{0000-0001-5254-7305} \and
Jiaowen Yang\inst{2} \and 
Jun Zhang\inst{3}}
\authorrunning{Tian et al.}
%
\institute{University of Toronto, Toronto, Ontario, Canada\\
\email{xiwei.tian@mail.utoronto.ca, tkl.wong@utoronto.ca}\and
Meta, Menlo Park, California, United States\\
\email{jiaowen@meta.com}
\and
University of Michigan, Ann Arbor, Michigan, United States\\
\email{junz@umich.edu}}
\maketitle              
\begin{abstract}
The $\lambda$-exponential family generalizes the standard exponential family via a generalized convex duality motivated by optimal transport. It is the constant-curvature analogue of the exponential family from the information-geometric point of view, but the development of computational methodologies is still in an early stage. In this paper, we propose a fixed point iteration for maximum likelihood estimation under i.i.d.~sampling, and prove using the duality that the likelihood is monotone along the iterations. We illustrate the algorithm with the $q$-Gaussian distribution and the Dirichlet perturbation.

\end{abstract}
\section{Introduction} \label{sec:intro}
Let a state space $\X$ and a reference measure $\nu$ on $\X$ be given. For example, we may let $\X$ be a Euclidean space and $\nu$ be the Lebesgue measure.

\begin{definition}[$\lambda$-exponential family]
Let $\lambda \in \R \setminus \{0\}$ be a fixed constant. A $d$-dimensional $\lambda$-exponential family dominated by $\nu$ is a parameterized probability density $p(x; \theta)$, $\theta \in \Theta \subset \R^d$, given by
\begin{equation} \label{eqn:lambda.exponential.family}
p(x; \theta) =  (1 + \lambda \theta \cdot F(x))_{+}^{\frac{1}{\lambda}} e^{-\varphi(\theta)}, \quad x \in \X,
\end{equation}
where $z_+ = \max\{z, 0\}$ is the positive part, $F = (F_1, \ldots, F_d): \X \rightarrow \R^d$ is a vector of statistics, and the primal potential function $\varphi(\theta)$ is defined by the normalization $\int_{\X} p(x; \theta) \dd \nu(x) = 1$.
\end{definition}

The $\lambda$-exponential family (defined by a {\it divisive normalization}) was introduced in \cite{WZ22} (also see \cite{W18}) as a natural extension of the standard {\it exponential family}
\begin{equation} \label{eqn:exp.family}
p(x; \theta) = e^{\theta \cdot F(x) - \phi(\theta)},
\end{equation}
which is recovered in \eqref{eqn:lambda.exponential.family} by letting $\lambda \rightarrow 0$. It also recovers, by switching to a {\it subtractive normalization} (see \cite[Section III.A]{WZ22}), the {\it $q$-exponential family} with $q = 1 - \lambda$ \cite{AO11,AOM12,N11}. The {\it $q$-Gaussian distribution} (see Example \ref{eg:q.Gaussian}) and the {\it Dirichlet perturbation} (see Section \ref{sec:examples}) are important examples of the $\lambda$-exponential family. The key idea is a generalized convex duality, called the {\it $\lambda$-duality}, which is a special case of the $c$-duality in optimal transport \cite{V03}. Whereas the cumulant generating function $\phi$ in \eqref{eqn:exp.family} is convex, when $\lambda < 1$ the function $\varphi$ in \eqref{eqn:lambda.exponential.family} can be shown (under suitable regularity conditions) to have the property that $\frac{1}{\lambda} (e^{\lambda \varphi} - 1)$ is convex. This generalized convexity is captured by the {\it $\lambda$-conjugate} 
\begin{equation} \label{eqn:lambda.conjugate}
f^{(\lambda)}(y) := \sup_x \left\{ \frac{1}{\lambda} \log (1 + \lambda x \cdot y) - f(x) \right\},
\end{equation}
in which the pairing function $ \frac{1}{\lambda} \log (1 + \lambda x \cdot y)$ is used in place of the usual inner product $x \cdot y$. Using the $\lambda$-duality based on \eqref{eqn:lambda.conjugate}, many information-geometric properties of the exponential family (obtained by applying the usual convex duality to the $\phi$ in \eqref{eqn:exp.family}) have natural analogues. In particular, while an exponential family is dually flat \cite{A16}, a $\lambda$-exponential family is dually projectively flat with constant sectional curvature $\lambda$. Further results about the $\lambda$-duality and its relations with the usual convex duality can be found in \cite{ZW21,ZW22}.


Several computational methodologies have been developed for the generalized exponential family (or its special cases such as the $t$-distribution) represented as \eqref{eqn:lambda.exponential.family} or as a $q$-exponential family: logistic regression \cite{DV11}, approximate inference \cite{DQV11}, classification \cite{FSS17}, nonlinear principal component analysis \cite{TW21}, dimension reduction \cite{ANK24}, online estimation \cite{KWR24}, maximum likelihood estimation and variational inference \cite{GCE24}, as well as compositional data analysis \cite{MCG25}. In this paper, we study maximum likelihood estimation of the $\lambda$-exponential family under i.i.d.~sampling, emphasizing the role played by the $\lambda$-duality. 

For a (regular) exponential family \eqref{eqn:exp.family}, it is well known (see \cite[Section 2.8.3]{A16}) that the maximum likelihood estimate (MLE) $\hat{\theta}$ given data points $x_1, \ldots, x_n$ (under independent sampling) is characterized by
\begin{equation} \label{eqn:exp.family.MLE}
\nabla \phi(\hat{\theta}) = \frac{1}{n} \sum_{i = 1}^n F(x_i).
\end{equation}
That is, the MLE $\hat{\eta}$ of the dual (expectation) parameter $\eta = \nabla \phi(\theta)$ is given simply by the sample mean of the sufficient statistic. For a $\lambda$-exponential family, the appropriate dual parameter is defined by the {\it $\lambda$-gradient} 
\begin{equation} \label{eqn:c.lambda.gradient}
\eta= \nabla^{(\lambda)} \varphi(\theta) := \frac{\nabla \varphi(\theta)}{1 - \lambda \nabla \varphi(\theta) \cdot \theta},
\end{equation}
which can be interpreted probabilistically as an escort expectation. The inverse mapping is given by $\theta = \nabla^{(\lambda)} \psi(\eta)$, where $\psi = \varphi^{(\lambda)}$ is the dual potential given by \eqref{eqn:c.lambda.psi} below and can be expressed as a R\'{e}nyi entropy \cite[Theorem III.14]{WZ22}. In Section \ref{sec:prelim}, we show that the MLE $\hat{\eta}$ for $\eta$ is a {\it convex combination} of $F(x_1), \ldots, F(x_n)$, where the weights depend on $\hat{\eta}$ and the data. This relation suggests a natural fixed-point iteration (see Algorithm \ref{alg:fixed.point}) to compute the MLE; it is different from the algorithm proposed recently in \cite{GCE24} which is based on optimizing a bound of the likelihood. In Theorem \ref{thm:main}, we show that when $\lambda < 0$ the likelihood is monotone along the iterations of our algorithm. In Section \ref{sec:examples}, we illustrate this algorithm with the Dirichlet perturbation model. Further investigation of this algorithm, as well as comparison with alternative approaches, will be addressed in a follow-up paper.



\section{MLE via fixed-point iteration} \label{sec:prelim}
Let a $\lambda$-exponential family \eqref{eqn:lambda.exponential.family} be given. Throughout this paper, we assume that the family satisfies the following regularity conditions. It can be verified that these conditions hold for the examples considered in this paper.

\begin{assumption} \label{ass:assumption}
We assume that $\lambda < 0$, $\Theta = \{ \theta \in \R^d:  \int (1 + \lambda \theta \cdot F(x))_+^{1/\lambda} \dd \nu(x) < \infty\}$ is the natural parameter set and  
\cite[Condition III.10]{WZ22} holds. Moreover, we assume that the dual parameter set $\Xi := \nabla^{(\lambda)} \varphi(\Theta)$ is convex and contains the common support $\mathcal{S} := \{x \in \X : p(x; \theta) > 0\}$ of the family.
\end{assumption}

Let $\varphi$ be the potential function in \eqref{eqn:lambda.exponential.family}. We define the {\it dual potential} $\psi$ by the $\lambda$-conjugate \eqref{eqn:lambda.conjugate} of $\varphi$:
\begin{equation} \label{eqn:c.lambda.psi}
\psi(\eta) := \varphi^{(\lambda)}(\eta) = \sup_{\theta' \in \Theta} \left\{ \frac{1}{\lambda} \log (1 + \lambda \theta' \cdot \eta) - \varphi(\theta') \right\}, \quad \eta \in \Xi.
\end{equation}
Under Assumption \ref{ass:assumption}, $\frac{1}{\lambda} (e^{\lambda \psi} - 1)$ is strictly convex on $\Xi$. Moreover, $\nabla^{(\lambda)} \varphi : \Theta \rightarrow \Xi$ is a diffeomorphism with inverse $\nabla^{(\lambda)} \psi$. We define the {\it dual parameter} $\eta$ by $\eta := \nabla^{(\lambda)} \varphi(\theta)$, so that $\theta = \nabla^{(\lambda)} \psi(\eta)$. From the definition of $\lambda$-conjugate we have the following analogue of the Fenchel-Young identity:
\begin{equation} \label{eqn:Fenchel.Young}
\varphi_{\lambda}(\theta) + \psi_{\lambda}(\eta) \equiv \frac{1}{\lambda} \log (1 + \lambda \theta \cdot \eta), \quad \eta = \nabla^{(\lambda)}\varphi(\theta),
\end{equation}
where $1 + \lambda \theta \cdot \eta$ can be shown to be strictly positive. The convexity of $\Xi$ is related to the $\lambda$-analogue of convex functions of Legendre type; see \cite[Section VI]{WZ22} for further discussion. 

Let $x_1, \ldots, x_n \in \X$ be $n$ i.i.d.~samples from the model \eqref{eqn:lambda.exponential.family} and write $y_i = F(x_i) \in \mathbb{R}^d$. We assume $x_i \in \mathcal{S}$ for all $i$ as otherwise the likelihood is zero. The log likelihood is then given by
\begin{equation} \label{eqn:log.likelihood}
\ell(\theta) = \sum_{i = 1}^n \log p(x_i; \theta) = \sum_{i = 1}^n \left(  \frac{1}{\lambda} \log (1 + \lambda \theta \cdot y_i) -  \varphi(\theta) \right).
\end{equation}

\begin{proposition} [First order condition] \label{prop:MLE}
Suppose Assumption \ref{ass:assumption} holds. If $\hat{\theta}$ is an MLE, i.e.~it maximizes $\ell(\theta)$, then $\hat{\eta} := \nabla^{(\lambda)} \varphi_{\lambda}(\hat{\theta})$ satisfies the relation
\begin{equation} \label{eqn:first.order.condition}
\hat{\eta} = \sum_{i = 1}^n w_i(\hat{\theta}) y_i,
\end{equation}
where the weights $w_i$ are defined by 
\begin{equation} \label{eqn:weight}
w_i(\theta) := \frac{1/(1 + \lambda \theta \cdot y_i)}{\sum_{j = 1}^n 1/(1 + \lambda \theta \cdot y_j)}, \quad i = 1, \ldots, n.
\end{equation}
\end{proposition}
\begin{proof}
Differentiating \eqref{eqn:log.likelihood} yields the first order condition $\nabla \varphi(\hat{\theta}) = \frac{1}{n}\sum_{i=1}^n \frac{y_i}{1+\lambda \hat{\theta} \cdot y_i}$. Note that
\begin{equation} \label{eqn:weights.normalization}
1 - \lambda \nabla \varphi(\hat{\theta}) \cdot \hat{\theta} = \frac{1}{n} \sum_{i = 1}^n \frac{1}{1 + \lambda \hat{\theta} \cdot y_i}.
\end{equation}
We obtain \eqref{eqn:first.order.condition} from the definition \eqref{eqn:c.lambda.gradient} of the dual parameter.
\end{proof}

When $\lambda \rightarrow 0$, \eqref{eqn:first.order.condition} reduces to \eqref{eqn:exp.family.MLE}. In general, when $\lambda \neq 0$  equation \eqref{eqn:first.order.condition} does not have a closed form solution. However, it suggests the following procedure for computing the MLE.

\begin{algorithm}[Fixed-point iteration] \label{alg:fixed.point}
Given data $y_1 = F(x_1), \ldots, y_n = F(x_n)$ and an initial guess $\theta(0) \in \Theta$, define $(\eta(k))_{k \geq 1} \subset \Xi$ and $(\theta(0))_{k \geq 1} \subset \Theta$ by
\begin{equation} \label{eqn:fixed.point.iterate}
\left\{
    \begin{array}{ll}
     \eta(k+1) &:= \sum_{i=1}^n w_i(\theta(k))y_i,\\
      \theta(k+1) &:= \nabla^{(\lambda)} \psi(\eta(k+1)).
    \end{array}
  \right.
\end{equation}
\end{algorithm}

More compactly, we may express \eqref{eqn:fixed.point.iterate} by the following dynamical system on the dual parameter space $\Xi$:
\begin{equation} \label{eqn:iteration}
\eta(k+1) := T(\eta(k)), \quad T(\eta) :=  \sum_{i = 1}^n w_i(\eta(k)) y_i,
\end{equation}
where by an abuse of notations we denote $w_i(\theta) = w_i(\nabla^{(\lambda)} \psi(\eta))$ also by $w_i(\eta)$. It is clear that if $\eta(k)$ already satisfies \eqref{eqn:first.order.condition}, then $\eta(k+1) = \eta(k)$. Since the convex hull of $y_1, \ldots, y_n$ is invariant under $T$, Brouwer's fixed-point theorem implies that a fixed point exists (uniqueness is left for future research).  We give a simple example to illustrate the algorithm.

\begin{figure}[t!]
		\centering
		\includegraphics[scale = 0.55]{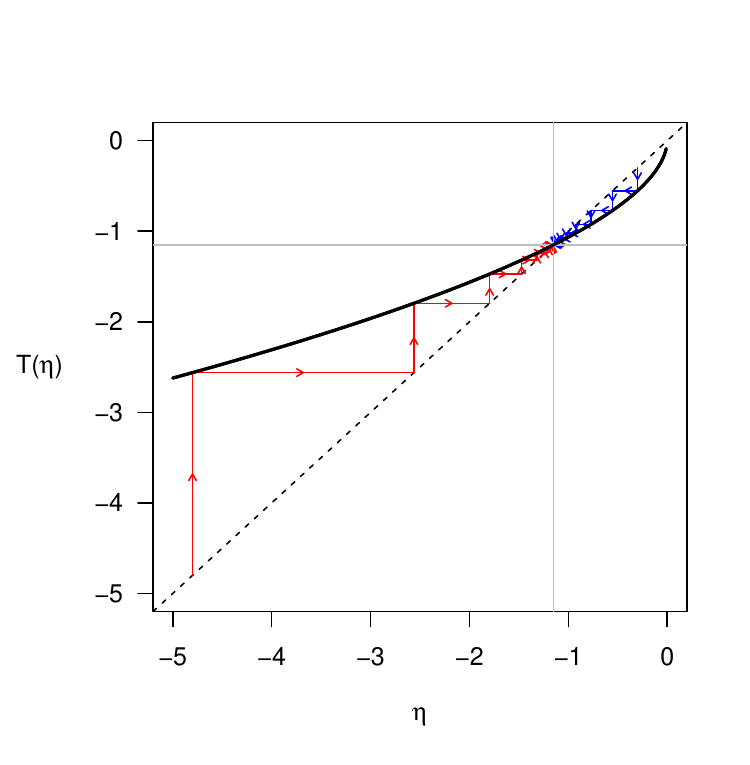}
        \vspace{-0.5cm}
		\caption{Two trajectories $(\eta(k))$ of Algorithm \ref{alg:fixed.point} for the $q$-Gaussian distribution in Example \ref{eg:q.Gaussian}. The solid curve shows the graph of $T(\eta)$ on $\Xi = (-\infty, 0)$ defined by \eqref{eqn:iteration} for the given data-set. The dashed line is the graph of the identity map. The fixed point (indicated by the cross) corresponds to the MLE.}
        \label{fig:qGaussian}
\end{figure}

\begin{example}[$q$-Gaussian distribution] \label{eg:q.Gaussian}
Let $\mathcal{X} = \mathbb{R}$ and $\nu$ be the Lebesgue measure. For $q \in (0, 3)$, the $q$-Gaussian distribution as a scale family can be expressed as a $\lambda$-exponential family with $\lambda = 1 - q \in (-2, 1)$:
\begin{equation} \label{eqn:q.Gaussian}
p(x; \theta) = (1 - \lambda \theta x^2)^{\frac{1}{\lambda}} e^{-\varphi(\theta)}, \quad x \in \mathbb{R},
\end{equation}
where $\theta \in \Theta = (-\infty, 0)$, $F(x) = x^2$ and $\varphi(\theta) = \frac{-1}{2} \log (-\theta) + C_{\lambda}$ for some constant $C_{\lambda}$ (see \cite[Example III.17]{WZ22} for details). The dual parameter is $\eta = \nabla^{(\lambda)} \varphi(\theta) = \frac{1}{2 + \lambda} \frac{1}{\theta} \in \Xi = (-\infty, 0)$, so $\theta = \nabla^{(\lambda)} \psi(\eta) = \frac{1}{2 + \lambda} \frac{1}{\eta}$. In Figure \ref{fig:qGaussian} we illustrate Algorithm \ref{alg:fixed.point} with a simulated data-set with $\lambda = -1.2$ and $n = 500$. The true value of $\theta$ is $-1$ and we show two trajectories with different initial values. In both cases, the iterates converge quickly to the MLE.
\end{example}


We are now ready to state the main result of the paper whose proof is a novel application of the $\lambda$-duality. In particular, we will use the $\lambda$-gradients and the strict convexity of $\frac{1}{\lambda} (e^{\lambda \psi} - 1)$. 

\begin{theorem}[Monotonicity of likelihood] \label{thm:main}
Under Assumption \ref{ass:assumption}, we have $\ell(\hat{\theta}(k+1)) > \ell(\hat{\theta}(k))$ unless $\hat{\theta}(k)$ already satisfies the fixed point condition \eqref{eqn:first.order.condition}.
\end{theorem}
\begin{proof}
We first express the weights $w_i$ in terms of the dual potential $\psi$. Since $\theta = \nabla^{(\lambda)} \psi(\eta)$, for each $i$ we have
\begin{equation} \label{eqn:proof.identity.1}
1 + \lambda \theta \cdot y_i = 1 + \lambda \frac{\nabla \psi(\eta)}{1 - \lambda \nabla \psi(\eta) \cdot \eta} \cdot y_i = \frac{1 + \lambda \nabla \psi(\eta) \cdot (y_i - \eta)}{1 - \lambda \nabla \psi(\eta) \cdot \eta}.
\end{equation}
Also note the identity
\begin{equation} \label{eqn:proof.identity.2}
1 - \lambda \nabla \psi(\eta) \cdot \eta = \frac{1}{1 + \lambda \theta \cdot \eta} = 1 - \lambda \nabla \varphi(\theta) \cdot \theta,
\end{equation}
which can be verified by a similar computation. Let $\Psi(\eta) := e^{\lambda \psi(\eta)}$ which is positive and (since $\lambda < 0$) strictly concave on $\Xi$. For each $i$, define 
\[
\kappa_i(\eta) := \Psi(\eta) + \nabla \Psi(\eta) \cdot (y_i - \eta).
\]
Using the chain rule, \eqref{eqn:proof.identity.1} and \eqref{eqn:proof.identity.2}, we have
\begin{equation} \label{eqn:proof.identity.3}
1 + \lambda \theta \cdot y_i = \frac{\kappa_i(\eta)}{\Psi(\eta)} (1 + \lambda \theta \cdot \eta) = \frac{\kappa_i(\eta)}{e^{\lambda \psi(\eta) - \log (1 + \lambda \theta \cdot \eta)}} = \kappa_i(\eta) e^{\lambda \varphi(\theta)}.
\end{equation}
Thus
\begin{equation}  \label{eqn:proof.identity.4}
\kappa_i(\eta) = p(x_i; \theta)^{\lambda} \quad \text{and} \quad w_i(\theta) = \frac{(\kappa_i(\eta))^{-1}}{\sum_{j = 1}^n (\kappa_j(\eta))^{-1}}.
\end{equation}
From \eqref{eqn:proof.identity.4}, we observe that
\begin{equation} \label{eqn:likelihood.kappa}
\ell(\theta) =  \frac{1}{\lambda} \log \left( \prod_{i = 1}^n \kappa_i(\eta) \right).
\end{equation}
Since $\lambda < 0$, maximizing the likelihood over $\theta \in \Theta$ is equivalent to minimizing the product $\prod_{i = 1}^n \kappa_i(\eta)$ over $\eta = \nabla^{(\lambda)} \psi(\theta) \in \Xi$.

Now we make the key observation. Consider the $k$-th iterate $(\theta(k), \eta(k))$. For any $\eta \in \Xi$, we have
\begin{equation} \label{eqn:main.observation}
\begin{split}
\sum_{i = 1}^n w_i (\theta(k)) \kappa_i(\eta) &= \sum_{i = 1}^n w_i(\theta(k)) \left( \Psi(\eta) + \nabla \Psi(\eta) \cdot (y_i - \eta) \right) \\
&= \Psi(\eta) + \nabla \Psi(\eta) \cdot \left( \sum_{i = 1}^n w_i(\theta(k)) y_i - \eta \right) \\
&=\Psi(\eta) + \nabla \Psi(\eta) \cdot (\eta(k+1) - \eta),
\end{split}
\end{equation}
where the last equality follows from \eqref{eqn:fixed.point.iterate}. Letting $\eta = \eta(k)$ and $\eta = \eta(k+1)$ gives
\[
\sum_{i = 1}^n w_i(\theta(k)) \kappa_i(\eta(k)) = \Psi(\eta(k)) + \nabla \Psi(\eta(k)) \cdot (\eta(k+1) - \eta(k))
\]
and
\[
\sum_{i = 1}^n w_i (\theta(k)) \kappa_i(\eta(k+1)) = \Psi (\eta(k+1)).
\]
Since $\eta(k) \neq \eta(k+1)$ by assumption, the strict concavity of $\Psi$ implies that
\begin{equation} \label{eqn:concavity.consequence}
\sum_{i = 1}^n w_i(\theta(k)) \kappa_i(\eta(k)) < \sum_{i = 1}^n w_i (\theta(k)) \kappa_i(\eta(k+1)).
\end{equation}
From \eqref{eqn:proof.identity.4}, the last inequality is equivalent to
\begin{equation*}
\sum_{i = 1}^n \frac{(\kappa_i(\eta(k)))^{-1}}{\sum_{j = 1}^n (\kappa_j(\eta(k)))^{-1}} \kappa_i(\eta(k+1)) < \sum_{i = 1}^n \frac{(\kappa_i(\eta(k)))^{-1}}{\sum_{j = 1}^n (\kappa_j(\eta(k)))^{-1}} \kappa_i(\eta(k)),
\end{equation*}
which leads to
\[
\frac{1}{n} \sum_{i = 1}^n \frac{\kappa_i(\eta(k+1))}{\kappa_i(\eta(k))} < 1.
\]
By the inequality of arithmetic and geometric means, we have
\[
\frac{\prod_{i = 1}^n \kappa_i (\eta(k+1))}{\prod_{i = 1}^n \kappa_i (\eta(k))} \leq \left( \frac{1}{n} \sum_{i = 1}^n \frac{\kappa_i(\eta(k+1))}{\kappa_i(\eta(k))} \right)^n < 1.
\]
From \eqref{eqn:likelihood.kappa} (and recalling that $\lambda < 0$), we have $\ell(\theta(k+1)) > \ell(\theta(k))$.
\end{proof}

\section{Dirichlet perturbation}  \label{sec:examples}
The {\it Dirichlet perturbation model} is a multiplicative analogue of the {\it normal location model}\footnote{By this we mean the model $Y = \theta + \epsilon$, where $\theta \in \mathbb{R}^d$ is the unknown mean and $\epsilon \sim N(0, \sigma^2 I)$ is the noise, and $\sigma > 0$ is treated as a nuisance parameter.} and is a key example of the $\lambda$-exponential family \cite{PW18b,WZ22}. Let $\Delta^{d}$ be the open unit simplex in $\mathbb{R}^{d+1}$ defined by
\[
\Delta^d = \left\{ p = (p^0, \ldots, p^d) \in (0, 1)^{d + 1}: p^0 + \cdots + p^d = 1 \right\}.
\]
In this section, we use superscripts to denote the components. On $\Delta^d$, introduce the {\it perturbation operation}
\begin{equation*} \label{eqn:perturbation}
p \oplus q := \left( \frac{p^0q^0}{\sum_{j = 0}^d p^jq^j}, \ldots, \frac{p^dq^d}{\sum_{j = 0}^d p^jq^j} \right),
\end{equation*}
which is the addition under the {\it Aitchison geometry} \cite{EPMB03}. This leads to the {\it difference operation}
\begin{equation*} 
p \ominus q := \left(  \frac{p^0/q^0}{\sum_{j = 0}^d p^j/q^j}, \ldots, \frac{p^d/q^d}{\sum_{j = 0}^d p^j/q^j} \right).
\end{equation*}
In particular, for any $p \in \Delta^d$, $p \ominus p = \left( \frac{1}{1 + d}, \ldots, \frac{1}{1 + d}\right)$ is the barycenter of the simplex (uniform distribution). Fix $\sigma > 0$ and let $D = (D^0, \ldots, D^d)$ be a Dirichlet random vector with parameters $\left( \frac{1}{\sigma(1 + d)}, \ldots, \frac{1}{\sigma(1 + d)}\right)$. We may regard $\sigma$ as a noise parameter. Consider the distribution of
\begin{equation} \label{eqn:Dirichlet.perturbation}
Q := p \oplus D,
\end{equation}
where $p \in \Delta^d$ is regarded as the parameter. In \cite[Proposition III.20]{WZ22}, it was shown that for $\sigma > 0$ fixed, the parameterized
distribution of $Q$ on $\mathcal{X} = \Delta^d$ can be expressed as a $d$-dimensional $\lambda$-exponential family with $\lambda = -\sigma < 0$,
\[
\theta := \left( \frac{p^0}{\lambda p^1}, \ldots, \frac{p^0}{\lambda p^d} \right) \quad \text{and} \quad F(q) := \left( \frac{q^1}{q^0}, \ldots, \frac{q^d}{q^0} \right).
\]
The potential function is $\varphi(\theta)= \frac{1}{\lambda (1 + d)} \sum_{i = 1}^d \log (-\theta^i)$, defined for $\theta \in \Theta := (-\infty, 0)^d$. The dual parameter is then given by
\begin{equation} \label{eqn:Dirichlet.eta}
\eta := \nabla^{(\lambda)} \varphi(\theta) = \frac{1}{\lambda} \left( \frac{1}{\theta^1}, \ldots, \frac{1}{\theta^d} \right) = \left( \frac{p^1}{p^0}, \ldots, \frac{p^d}{p^0} \right) \in \Xi := (0, \infty)^d,
\end{equation}
which is independent of $\sigma$ or $\lambda$. In particular, we may recover $p$ from $\eta$ by
\begin{equation} \label{eqn:Dirichlet.eta.to.p}
p = (p^0, p^1, \ldots, p^d) = \left( \frac{1}{1 + \sum_{j = 1}^d \eta^j}, \frac{\eta^1}{1 + \sum_{j = 1}^d \eta^j}, \ldots, \frac{\eta^d}{1 + \sum_{j = 1}^d \eta^j} \right).
\end{equation}
With \eqref{eqn:Dirichlet.eta.to.p}, we can express the update \eqref{eqn:iteration} in terms of the {\it simplex parameter} $p(k)$ rather than the dual parameter $\eta(k)$.

In the following proposition, we show that the update for the Dirichlet perturbation is, in fact, independent of $\lambda$. That is, the algorithm works the same way even if $\sigma = -\lambda$ is {\it unknown}. This is analogous to maximum likelihood estimation of the normal location model, where the sample mean does not depend on the value of the noise. This is not the case in general.

\begin{proposition}
Let $q_1, \ldots, q_n \in \Delta^d$ be $n$ i.i.d.~samples from the Dirichlet perturbation model \eqref{eqn:Dirichlet.perturbation}. Under the algorithm \eqref{eqn:iteration}, we have
\begin{equation} \label{eqn:Dirichlet.p.update}
p(k+1) = p(k) \oplus \left( \frac{1}{n} \sum_{i = 1}^n (q_i \ominus p(k)) \right).
\end{equation}
We emphasize that $\frac{1}{n} \sum_{i = 1}^n$ is the Euclidean (not Aitchison) average on $\Delta^d$.
\end{proposition}
\begin{proof}
We plug \eqref{eqn:Dirichlet.eta} and \eqref{eqn:Dirichlet.eta.to.p} into \eqref{eqn:iteration} and simplify using the simplex operations. The details are omitted due to space constraints.
\end{proof}



\begin{figure}[t!]
		\centering
		\includegraphics[scale = 0.48]{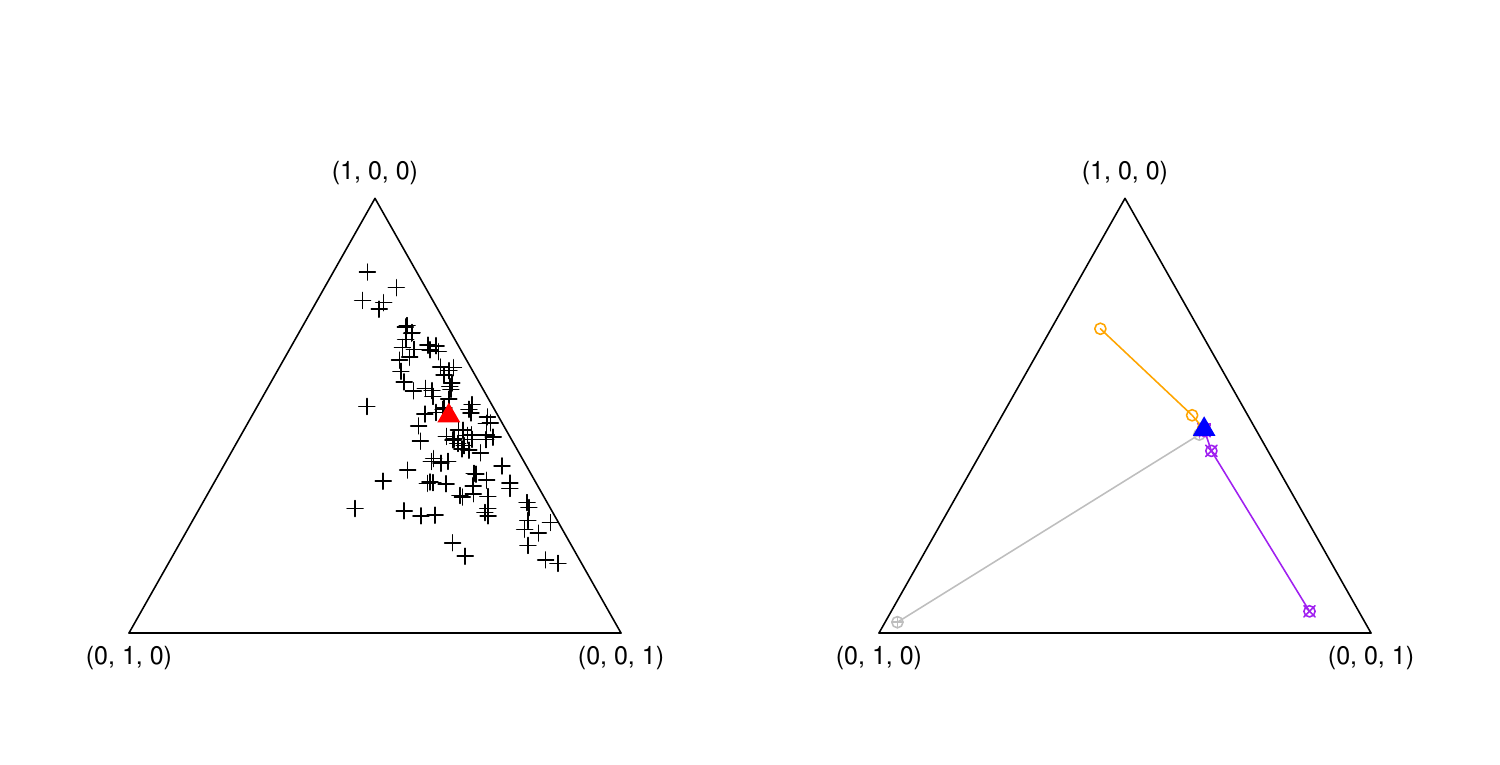}
        \vspace{-1.2cm}
		\caption{{\bf Left:} True $p$ (red {\color {red} $\blacktriangle$}) and samples ($+$) from the Dirichlet perturbation model. {\bf Right:} Three trajectories $(p(k))_{k \geq 0}$ of the algorithm \eqref{eqn:Dirichlet.p.update} with different initial values. They all converge quickly to the MLE (blue {\color {blue} $\blacktriangle$}).}
        \label{fig:Dir}
\end{figure}


\begin{example} \label{eg:Dir}
We consider a simulated data-set from the Dirichlet perturbation model with $d = 2$, $p = (0.1, 0.4, 0.5)$, $\sigma = 0.1$ and $n = 100$ (chosen for visualization purposes; the algorithm works for much larger values of $d$ and $n$). In Figure \ref{fig:Dir} we plot the data and the output of the algorithm \eqref{eqn:Dirichlet.p.update} with several initial values. In all cases, the iterates converge quickly to the MLE. In practice, we may initialize $p(0)$ by the sample mean $\frac{1}{n} \sum_{i = 1}^n q_i$ as it tends to be close to the MLE.
\end{example}

\begin{credits}
\subsubsection{\ackname} The research of T.-K.~L.~Wong is partially supported by an NSERC Discovery Grant (RGPIN-2019-04419).

\subsubsection{\discintname}
The authors have no competing interests to declare that are
relevant to the content of this article.
\end{credits}
%
%
%
\bibliographystyle{splncs04}
\bibliography{references}

\end{document}